\documentclass[reqno]{amsart}
\usepackage[scale=0.75, centering, headheight=14pt]{geometry}
\usepackage[latin1]{inputenc}
\usepackage[T1]{fontenc}
\usepackage{lmodern}
\usepackage[english]{babel}

\usepackage{tikz}
\usepackage{bbm}
\usepackage{amsmath,amssymb,amsfonts,amsthm}
\usepackage{mathtools,accents}
\usepackage{mathrsfs}
\usepackage{xfrac}
\usepackage{array} 
\usepackage{aliascnt}
\usepackage{booktabs} 
\usepackage{array} 

\usepackage{verbatim} 
\usepackage{subfig} 

\usepackage{mathrsfs, dsfont}
\usepackage{amssymb}
\usepackage{amsthm}
\usepackage{amsmath,amsfonts,amssymb,esint}
\usepackage{graphics,color}
\usepackage{enumerate}
\usepackage{mathtools,centernot}

\usepackage{microtype}
\usepackage{paralist} 
\usepackage{cases}
\usepackage[initials]{amsrefs}
\allowdisplaybreaks

\usepackage{braket}
\usepackage{bm}

\usepackage[citecolor=blue,colorlinks]{hyperref}
\addto\extrasenglish{}

\usepackage{enumerate}
\usepackage{xcolor}
\usepackage{aliascnt}
\makeatletter
\def\newaliasedtheorem#1[#2]#3{
  \newaliascnt{#1@alt}{#2}
  \newtheorem{#1}[#1@alt]{#3}
  \expandafter\newcommand\csname #1@altname\endcsname{#3}
}
\makeatother

\theoremstyle{plain}
\newtheorem{theorem}{Theorem}[section]
\newaliasedtheorem{lemma}[theorem]{Lemma}
\newaliasedtheorem{prop}[theorem]{Proposition}
\newaliasedtheorem{claim}[theorem]{Claim}
\newaliasedtheorem{corollary}[theorem]{Corollary}
\newtheorem*{Con}{Conjecture}
\theoremstyle{remark}

\theoremstyle{definition}
\newaliasedtheorem{definition}[theorem]{Definition}
\newaliasedtheorem{example}[theorem]{Example}

\theoremstyle{remark}
\newaliasedtheorem{remark}[theorem]{Remark}

\numberwithin{equation}{section}

\def\eps{\varepsilon}
\def\tr{\text{tr }}
\def\R{\mathbb R}
\def\C{{\mathbb C}}
\def\N{{\mathbb N}}
\def\Z{{\mathbb Z}}
\def\T{{\mathbb T}}

\DeclareMathOperator{\diver}{div}
\DeclareMathOperator{\curl}{curl}
\DeclareMathOperator{\supp}{supp}

\newcommand*{\RR}{\ensuremath{\mathcal{R}}}

\DeclareMathOperator{\Id}{Id}
\DeclareMathOperator{\dist}{d}
\title{Sharp energy regularity and typicality results for H\"older solutions of incompressible Euler equations}

\author[L. De Rosa and  R. Tione]{Luigi De Rosa \and Riccardo Tione}
\address{Luigi De Rosa 
\hfill\break  \'Ecole Polytechnique F\'ed\'erale de Lausanne, Institute of Mathematics, Station 8, CH-1015 Lausanne, Switzerland.}
\email{luigi.derosa@epfl.ch}
\address{Riccardo TIone  
\hfill\break  Institut f\"ur Mathematik, Universit\"at Z\"urich, Winterthurerstrasse 190, CH-8057 Zurich, Switzerland.}
\email{riccardo.tione@math.uzh.ch}

\begin{document}

\maketitle

\begin{abstract}
This paper is devoted to show a couple of typicality results for weak solutions $v\in C^\theta$ of the Euler equations, in the case $\theta<1/3$. It is known that convex integration schemes produce wild weak solutions that exhibit anomalous dissipation of the kinetic energy $e_v$. We show that those solutions are typical in the Baire category sense. From \cite{Is15}, it is know that the kinetic energy $e_v$ of $\theta$-H\"older continuous weak solution $v$ of the Euler equations satisfy $ e_v\in C^{\frac{2\theta}{1-\theta}}$. As a first result we prove that solutions with that behavior are a residual set in suitable complete metric space $X_\theta$, that is contained in the space of all $C^\theta$ weak solutions, whose choice is discussed at the end of the paper. More precisely we show that the set of solutions $v\in X_\theta$ with $e_v \in C^{\frac{2\theta}{1-\theta}}$ but not to $\bigcup_{p\ge 1,\eps >0}W^{\frac{2\theta}{1-\theta} + \eps,p}(I)$ for any open $I \subset [0,T]$, are a residual set in $X_\theta$. This, in particular, partially solves \cite[Conjecture 1]{IsSu}. We also show that smooth solutions form a nowhere dense set in the space of all the $C^\theta$ weak solutions. The technique is the same and what really distinguishes the two cases is that in the latter there is no need to introduce a different complete metric space with respect to the natural one.
\end{abstract}
\par
\medskip\noindent
\textbf{Keywords:} incompressible Euler equations, H\"older solutions, energy regularity, convex integration, Baire category.
\par
\medskip\noindent
{\sc MSC (2010): 35Q31 - 35D30 - 76B03 - 26A21.
\par
}
\section{Introduction}\label{sec:intro}
In the spatial periodic setting $\T^3=\R^3 / \Z^3$, we consider the incompressible Euler equations
\begin{equation}\label{E}
\left\{\begin{array}{l}
\partial_t v+  \diver(v \otimes v)+\nabla p =0\\ 
 \diver v = 0\,
\end{array}\right.\qquad \mbox{in }  \T^3\times [0,T] 
\end{equation}
where $v:  \T^3\times [0,T]   \rightarrow \R^3$ represents the velocity of an incompressible fluid, $p:\T^3\times [0,T]  \rightarrow \R$ is the hydrodynamic pressure, with the constraint $\int_{\T^3}p\, dx =0$ which guaranties its uniqueness.\\

A weak solution of the system \eqref{E} is a vector field $v\in L^2(\T^3\times [0,T];\R^3)$ such that 
$$
\int_0^T\int_{\T^3}\left( v\cdot \partial_t \varphi+ v\otimes v : \nabla \varphi \right)\,dx dt=0,
$$
for all $\varphi\in C^\infty_c(\T^3\times (0,T);\R^3)$ such that $\diver \varphi=0$. The pressure does not appear in the weak formulation because it can be recovered as the unique $0$-average solution of 
$$
-\Delta p=\diver \diver (v\otimes v).
$$
Multiplying by $v$ the first equation in \eqref{E} and integrating by parts on $\T^3$, one gets that, at least for smooth solutions, 
$$
\frac{d}{dt }e_v(t):=\frac{d}{dt}\int_{\T^3} |v|^2(x,t)\, dx=0, \qquad \forall t\in [0,T].
$$
For weak solutions $v\in L^\infty((0,T);C^\theta(\T^3))$ it is known, and was previously conjectured by Lars Onsager, that the threshold for the energy conservation is $\theta=1/3$. The first proof of the conservation in the range $\theta>1/3$ was given in \cite{CET}, while in \cite{Is} P. Isett proved the existence of dissipative solutions for any $\theta<1/3$ using the convex integration techniques introduced by C. De Lellis and L. Sz\'ekelyhidi in \cite{DS}.\\

As observed in \cite{Is15}, given any solution $v\in L^\infty((0,T);C^\theta(\T^3))$, it can be shown that the associated kinetic energy $e_v$ satisfies
\begin{equation}\label{en_reg}
\left| e_v(t)-e_v(s)\right| \leq C \left| t-s\right|^{\frac{2\theta}{1-\theta}} \qquad \forall t,s\in [0,T],
\end{equation}
which in particular implies the conservation if $\theta>1/3$, but also shows a peculiar H\"older regularity of the energy (see also \cite{CD} for an alternative proof). Throughout the work, we will sometimes use the shorter notation
\[
\theta^*:= \frac{2\theta}{1-\theta}.
\]

P. Isett and S.-J. Oh conjectured in \cite[Conjecture 1]{IsSuCON} that this exponent is optimal in the following sense

\begin{Con}
For any $\theta < \frac{1}{3}$, there exists a solution to \eqref{E} in the class $v\in C^\theta(\R\times \T^n)$ whose energy profile $e(t)$ fails to have any regularity above the exponent $\frac{2\theta}{1 - \theta}$, in the sense that $e_v(t) \notin W^{\frac{2\theta}{1-\theta}+\eps,p}(I)$, for every $\eps > 0$, $p \ge 1$ and every open time interval $I \subset\R$. Furthermore, the set of all such solutions v with the above property is residual (in the sense of
category) within the space of all $ C^\theta(\R\times \T^n)$ weak solutions, endowed with the topology from the $C^\theta$ norm.
\end{Con}

In this paper we solve this conjecture in a slightly smaller space than $C^\theta$. This is due to some technical reasons and we postpone the discussion about this choice at the end of the introduction. \\

The first of our main results is the following

\begin{theorem}\label{t_exist}
Fix $\gamma > 0$ and $\theta\in (0,1/3)$ such that $\frac{2\theta}{1-\theta}+\gamma<1$. For every strictly positive $e \in C^{\frac{2\theta}{1-\theta}+\gamma}([0,T])$, there exists a vector field $v \in C^\theta( \T^3\times[0,T])$ that solves \eqref{E} in the distributional sense and such that
\[
 e(t)=\int_{\T^3}|v|^2(x,t)\, dx, \qquad \forall t \in [0,T].
\]
\end{theorem}
The proof of this result follows closely the one of \cite{BDSV}. In particular, our Theorem \ref{t_exist} states the same conclusion of \cite[Theorem 1.1]{BDSV}, except for the fact that we are dropping the hypothesis on the smoothness of the function $e$. We remark that such sharpness of the  energy regularity was first proven in \cite{IsSuCON,IsSu} for any $\theta\in (0,1/5)$. Here we extend the result to the whole range $(0,1/3)$, even though it must be noted that in \cite{IsSuCON,IsSu} the energy profile is allowed to vanish, while in the scheme of \cite{BDSV}, and thus in ours, this is not. A small refinement of Theorem \ref{t_exist}, coupled with a suitable $h$-principle, also yields that weak solutions $v\in C^\theta(\T^3\times [0,T])$ belonging to a proper, yet quite large, subset of the space of all weak solutions have typically a kinetic energy $e_v$ which is not more regular than $C^{\frac{2\theta}{1-\theta}}([0,T])$. To state it in a more precise way we set
\begin{equation}\label{Xtheta}
X_\theta=\overline{\left\{ v \in \bigcup_{\theta '>\theta}C^{\theta '}(\T^3\times[0,T])  \, : \, v\, \text{ weakly solves } \eqref{E}\right\}}^{\|\cdot \|_{C^\theta_{x,t}}},
\end{equation}
endowed with the distance
\begin{equation}\label{distanza}
\dist (u,v)=\|u-v\|_{C^\theta_{x,t}}.
\end{equation}
It is clear that $(X_\theta, \dist)$ is a complete metric space.  We also  define  
\[
W^{\theta^*} =\bigcup_{I\subset [0,T]}\bigcup_{p \ge 1}\bigcup_{\eps > 0}W^{\theta^* + \eps,p}(I)
\]
and
\begin{equation}\label{Ytheta}
Y_\theta=\left\{ v \in X_\theta \, : \, e_v\in C^{\theta^*}([0,T])\setminus W^{\theta^*} \right\}.
\end{equation}
We prove the following 
\begin{theorem}\label{t_baire}
For any $\theta\in (0,1/3)$, the set $Y_\theta$ is residual in $X_\theta$.
\end{theorem}

Baire Theorem asserts that a complete metric space is not meager. Therefore, the previous Theorem yields some immediate corollaries.  

First, it implies that the typical solution in $X_\theta$ is not of bounded variation, thus not monotonic, in any open subset of $[0,T]$. Thus, Theorem \ref{t_baire} shows a very irregular behaviour of the energy of solutions, in sharp contrast with the conservation of the energy in the case $\theta > 1/3$. We refer the reader to \cite{IsSuCON,IsSu} for further discussions.

A second immediate corollary of Theorem \ref{t_baire} is that, for every $\theta \in (0,1/3)$, there exists a weak solution $v$ of \eqref{E} such that $e_v \in C^{\theta^*}([0,T])$ but $e_v \notin C^{\theta^* + \gamma}([0,T])$, for any $\gamma > 0$. Let us note in passing that this also yields a weak $C^{\theta}(\T^3\times [0,T])$ solution of \eqref{E} that is not in $C^{\theta + \gamma}(\T^3\times [0,T])$, for any $\gamma$.
Indeed, from \eqref{en_reg} it is clear that $Y_\theta$ can not contain solutions $v$ that are more H\"older regular than $C^\theta(\T^3\times [0,T])$. While the residuality property implies that the kinetic energy of many  $C^\theta(\T^3\times [0,T])$ solutions enjoys the sharp regularity \eqref{en_reg}, it must be noted that $X_\theta$ might not contain all the $C^\theta(\T^3\times [0,T])$ solutions of Euler, since in general not all the $C^\theta(\T^3\times [0,T])$ functions can be obtained as limit of more regular ones. In particular it is not clear to the authors if the same statement is true if one considers as a complete metric space in Theorem \ref{t_baire}  all the $ C^\theta(\T^3\times [0,T]) $ solutions of \eqref{E}, endowed with the same distance $\dist (u,v):=\|u-v\|_{C^\theta_{x,t}}$. This would solve \cite[Conjecture 1]{IsSuCON} completely. We refer the reader to Section \ref{comm} for a more detailed discussion on this problem.\\

Another natural question is about the topological property of the smooth solutions in this setting. To this end we define $S$ to be the set of all smooth solutions of \eqref{E}, and similarly as before we also set
$$
C_\theta=\left\{ v\in C^\theta(\T^3 \times [0,T]) \, : \, v \, \textit{ weakly solves } \eqref{E} \right\},
$$
together with the natural distance \eqref{distanza}. Note at first that, as a corollary of Theorem \ref{t_baire}, one already gets that $S \subset Y_\theta^c$  which obviously implies that $S$ is a meager set in $X_\theta$. However, in this case, a stronger result can be proved

\begin{theorem}\label{t_smooth}
For any $\theta\in(0,1/3)$, the set $S$ of all smooth solutions of \eqref{E} is nowhere dense in $C_\theta$.
\end{theorem}

We recall that in a complete metric space,  a nowhere dense set is a set whose closure has empty interior. Thus Theorem \ref{t_smooth} is stronger with respect to the corollary that Theorem \ref{t_baire} would give from two points of view. Firstly, every nowhere dense set is also meager. Secondly, the corresponding topological property it is proved in a larger, and also more natural, space $C_\theta$.

\subsection*{Aknowledgements}

The authors would like to thank Camillo De Lellis for his interest in this problem and the useful discussions about it.

\section{Notations and main inductive Proposition}\label{main_prop}
Along the paper, we will consider the flat torus $\T^3$ as spatial domain, identifying it with the 3-dimensional cube $[0,1]^3 \subset \R^3 $. Thus for any $f:\T^3\to\R^3$ we will  always work with its periodic extension to the whole space.
We will follow the construction given in \cite{BDSV} dropping the hypothesis of the smoothness of the energy. We start by introducing the notation and some basic properties of the incompressible Euler equations.

\subsection{Notation}

In the following $N\in \N$, $\alpha\in (0,1)$ and $\kappa$ is a multi-index. We introduce the usual (spatial) 
H\"older norms as follows.
First of all, the supremum norm is denoted by $\|f\|_0:=\sup_{\T^3\times [0,T]}|f|$. We define the H\"older seminorms 
as
\begin{equation*}
\begin{split}
[f]_{N}&=\max_{|\kappa|=N}\|D^{\kappa}f\|_0\, ,\\
[f]_{N+\alpha} &= \max_{|\kappa|=N}\sup_{x\neq y, t}\frac{|D^{\kappa}f(x, t)-D^{\kappa}f(y, t)|}{|x-y|^{\alpha}}\, ,
\end{split}
\end{equation*}
where $D^\kappa$ are space derivatives only.
The H\"older norms are then given by
\begin{eqnarray*}
\|f\|_{N}&=&\sum_{j=0}^N[f]_j\\
\|f\|_{N+\alpha}&=&\|f\|_N+[f]_{N+\alpha}.
\end{eqnarray*}
Moreover, we will write $[f (t)]_\alpha$ and $\|f (t)\|_\alpha$ when the time $t$ is fixed and the norms are computed for the restriction of $f$ to the $t$-time slice. On the other hand we will explicitly write $\| f\|_{C^\alpha_{x,t}} $ when the H\"older norm is computed in both the space and time variables.\\
We also recall that, for an interval $I \subset \R$, and for $\theta \in (0,1)$, $p \ge 1$, 
\[
u \in W^{\theta,p}(I)
\]
if and only if $u \in L^p(I)$ and
\[
[u]_{W^{\theta,p}(I)}:=\left(\int_{I\times I}\frac{|u(x) - u(y)|^p}{|x - y|^{1 + p\theta}}dxdy\right)^{\frac{1}{p}} < + \infty.
\]
Let  $\varphi \in C^ \infty_c (B_1(0))$ be a standard non negative  kernel such that  $\int_{B_1(0)} \varphi(x) dx=1$. For any $ \delta>0 $ we define  $ \varphi_\delta:=\delta^{-3} \varphi(\frac{x}{\delta})$ and we denote the mollifications of a function $f$ as usal as
\[
f_\delta:=f*\varphi_\delta.
\] 
We recall the following standard estimates on the mollification of both H\"older continuous functions and vector fields.
\begin{prop}\label{p:moll}
For any $\theta \in(0,1)$ we have
\begin{equation}\label{mollest2}
\|f_\delta -f \|_{0} \leq \delta^{\theta} [f]_{\theta}.
\end{equation}
Moreover, for any $N\geq 0$, there exists a constant $C>0$ depending on $N$, such that 
\begin{align}
  \| f_\delta * f_\delta -(f * f)_\delta \|_{N} &\leq C \delta^{2\theta-N} [f]^2_{\theta}\,,\label{mollest4} \\
\| \ f_\delta\|_{N+1} &\leq C \delta^{\theta-N-1} [f]_{\theta}. \, \label{mollest3} 
\end{align}
\end{prop}

Given a metric space $(X,\dist)$, a subset $Y\subset X$ is said to be residual if its complement $Y^c$ is contained in a countable union of closed sets with empty interior. The set $Y^c$ is then called meager.\\
\\
Finally, we also recall that equations \eqref{E} are invariant under the following transformation
\begin{equation}\label{recaling}
v(x,t)\mapsto v_\Gamma (x,t) := \Gamma v(x,\Gamma t) \quad \text{ and } \quad p(x,t)\mapsto p_\Gamma (x,t) :=\Gamma^2p(x,\Gamma t),
\end{equation}
for any $\Gamma>0$, meaning that if $(v,p)$ solves \eqref{E} in $ \T^3\times [0,T]$  then $(v_\Gamma, p_\Gamma)$ solves \eqref{E} in $ \T^3\times [0,T/ \Gamma].$

\subsection{Inductive proposition}
As said, the proof is based on a modification of the convex integration scheme of \cite{BDSV}, that we are now going to explain.
\\
\\
Let $q\geq 0$ be a natural number. At a given step $q$ we assume to have a smooth triple $(v_q,p_q,\mathring{R}_q)$ solving the Euler-Reynolds system, namely such that
\begin{equation}\label{NSR}
\left\{\begin{array}{l}
\partial_t v_q + \diver (v_q\otimes v_q) + \nabla p_q =\diver\mathring{R}_q\\ \\
\diver v_q = 0\, ,
\end{array}\right.
\end{equation} 
to which we add the constraints 
\begin{align}
&\tr \mathring{R}_q=0 \,,\label{e:trace_free} \\
\int_{\T^3} p_q& (x,t)\, dx = 0\,.\label{e:press_const}
\end{align}

To measure the size of the approximate solution $v_q$ and the error $\mathring{R}_q$, we use a frequency $\lambda_q$ and an amplitude $\delta_q$, defined through these relations:
\begin{align}
\lambda_q &= 2\pi \lceil a^{(b^q)}\rceil, \label{e:freq_def}\\
\delta_q &=\lambda_q^{-2\beta}, \label{e:size_def}
\end{align}
where $\lceil x\rceil $ denotes the smallest integer $n\geq x$, $a>1$ is a  large parameter, $b>1$ is close to $1$ and $0<\beta<1/3$. The parameters $a$ and $b$ will depend on $\beta$ and on other quantities. We proceed by induction, assuming the estimates
\begin{align}
\|\mathring R_q\|_{0}&\leq  \delta_{q+1}\lambda_q^{-3\alpha}\label{e:R_q_inductive_est}\\
\|v_q\|_1&\leq M \delta_q^{\sfrac12}\lambda_q\label{e:v_q_inductive_est}\\
\|v_q\|_0 & \leq 1- \delta_q^{\sfrac12}\label{e:v_q_0}\\
\delta_{q+1}\lambda_q^{-\alpha} &\leq e(t)-\int_{\T^3}|  v_q|^2\,dx\leq \delta_{q+1}\label{e:energy_inductive_assumption}
\end{align}
where $0 < \alpha  < 1$ is a small parameter to be chosen suitably, in dependence of $\beta$ and other quantities, and $M$ is a universal constant.\\
\\
We now state the main inductive proposition 
\begin{prop}\label{p:main} 
There exists a universal constant $M$ with the following property. Let $0<\beta<\eta<1/3$, $E> 0$, and
\begin{equation}\label{e:b_beta_rel}
1<b<\sqrt{\frac{\eta^*}{\beta^*}}.
\end{equation}
Then there exists an $\alpha_0$ depending on $\beta$, $\eta$ and $b$, such that for any $0<\alpha<\alpha_0$ there exists an $a_0$ depending on $\beta$, $b$, $\alpha$, $\eta$, $E$  and $M$, such that for any $a\geq a_0 $ the following holds: given a  triple $(v_q, p_q,\mathring R_q)$ solving \eqref{NSR}-\eqref{e:press_const} and satisfying the estimates \eqref{e:R_q_inductive_est}--\eqref{e:energy_inductive_assumption} for some strictly positive $e\in C^{\eta^*}([0,T])$ with $$\|e\|_{\eta^*} \le E,$$ there exists a solution  $(v_{q+1}, p_{q+1}, \mathring R_{q+1})$ to \eqref{NSR}-\eqref{e:press_const} satisfying \eqref{e:R_q_inductive_est}--\eqref{e:energy_inductive_assumption} for the same function $e$ with $q$ replaced by $q+1$. Moreover, we have 
\begin{equation}
\|v_{q+1}-v_q\|_0+\frac{1}{\lambda_{q+1}}\|v_{q+1}-v_q\|_1 
\leq M\delta_{q+1}^{1/2}\label{e:v_diff_prop_est}.
\end{equation}
\end{prop}
The reader may notice that there are four main differences with respect to \cite[Proposition 2.1]{BDSV}. First of all the statement is fomulated in a slightly different way than in \cite[Proposition 2.1]{BDSV}, in order to highlight the fact that the parameter $a_0$ is uniform once one has chosen the $C^{\eta^*}([0,T])$ norm of $e$. Moreover, we drop the smoothness hypothesis on the function $e$, we allow the parameter $a_0$ to depend on $E$ and finally we suppose in \eqref{e:b_beta_rel} a different relation between the parameters $b$ and $\beta$.  Notice that our relation \eqref{e:b_beta_rel} is more restrictive than the one used in \cite{BDSV}, indeed we have
\begin{equation}\label{b_comecamillo}
1<b<\sqrt{\frac{\eta^*}{\beta^*}}<\sqrt{\frac{1}{\beta^*}}=\sqrt{\frac{1-\beta}{2\beta}}<\frac{1-\beta}{2\beta}.
\end{equation}

\section{Proof of the main Theorems}
In this section we prove our two main theorems. As in \cite{BDSV}, the proof of Theorem \ref{t_exist} is a direct consequence of Proposition \ref{p:main} and we are going to prove it for the reader's convenience. Theorem \ref{t_baire} will still be an application of the iterative proposition. Indeed, through a $h$-principle comparable to \cite[Theorem 1.3]{BDSV}, we will be able to write the set $Y^c_\theta$ as a countable union of closed set with empty interior.
\subsection{Proof of Theorem \ref{t_exist}} First of all, fix $\gamma, \theta$ and $e$ as in the statement of the theorem. In order to apply Proposition \ref{p:main} we choose  $\eta\in (0,1/3)$ to be the only solution of $\eta^*=\theta^*+\gamma$ and $\beta$ such that  $\theta<\beta<\eta$. Consequently we also fix the parameters $b$ and $\alpha$ appearing in the statement of Proposition \ref{p:main}, the first satisfying \eqref{e:b_beta_rel} and the second lower than the threshold $\alpha_0$. As done in \cite[Proof of Theorem 1.1]{BDSV}, by using the invariance of the Euler equations under the rescaling \eqref{recaling} we can further assume that the rescaled energy profile, say $\tilde e$, satisfies 
$$
\delta_1\lambda_0^{-\alpha}\leq \inf_t \tilde e(t)\leq \sup_t \tilde e(t)\leq \delta_1.
$$
Then we can apply inductively Proposition \ref{p:main} starting with the triple $(v_0,p_0,\mathring{R}_0)=(0,0,0) $. Indeed $v_0$ and $\mathring{R}_0$ trivially satisfy estimates \eqref{e:R_q_inductive_est}-\eqref{e:v_q_0} and by the rescaling on the energy we also get \eqref{e:energy_inductive_assumption} for $q=0$. By \eqref{e:v_diff_prop_est} we have 
\begin{equation}\label{stimalimite}
\sum_{q=0}^{\infty} \|v_{q+1}-v_q\|_\theta \lesssim \sum_{q=0}^{\infty} \|v_{q+1}-v_q\|_0^{1-\theta} \|v_{q+1}-v_q\|_1^\theta \lesssim  \sum_{q=0}^{\infty}  \delta_{q+1}^{\sfrac12}\lambda_{q+1}^\theta\lesssim  \sum_{q=0}^{\infty}  \lambda_{q+1}^{\theta-\beta}<\infty
\end{equation}
and hence $v_q$ converges in $C^0([0,T/\Gamma];C^\theta(\T^3))$ to a function $v$. Moreover, by \cite[Theorem 1.1]{CD}, we have that $v\in C^\theta(\T^3\times [0,T/\Gamma])$. By taking the divergence of the first equation in  \eqref{NSR}, we get that $p_q$ is the unique $0$-average solution of 
$$
-\Delta p_q=\diver \diver (v_q \otimes v_q-\mathring{R}_q)
$$
and since $v_q\otimes v_q -\mathring{R}_q\rightarrow v\otimes v$ uniformly, $p_q$ is also converging to some function $ p$ in $L^r(\T^3\times [0,T/\Gamma])$, for any $r<\infty$. Hence, it is clear that the limit couple $(v,p)$ solves \eqref{E} in the distributional sense in $\T^3\times [0,T/\Gamma]$. Finally, by \eqref{e:energy_inductive_assumption}, as $q\rightarrow \infty$, we also get
$$
\tilde e(t)=\int_{\T^3} |v|^2(x,t) \,dx \quad \forall t\in [0,T/\Gamma].
$$
The proof of the theorem is concluded by undoing the rescaling \eqref{recaling} and mapping back to the original time interval $[0,T]$.

\subsection{Proof  of Theorem \ref{t_baire}} We want to show that $Y_\theta^c$ is meager in $X_\theta$. Let us enumerate the intervals with rational endpoints inside $[0,T]$, $(I_r)_{r \in \N}$, and let $(q_s)_s$ be a countable and dense subset of $[1,+\infty)$. By \eqref{Ytheta} we can write
$$
Y^c_\theta=\bigcup_{m,n,r,s\in \N} C_{m,n,r,s},
$$
where 
$$
C_{m,n,r,s}\doteq\left\{ v\in X_\theta \,: \, \|e_v\|_{W^{\theta^*+\frac{1}{m},q_s}(I_r)}\leq n  \right\}.
$$
It is easily seen that $C_{m,n,r,s}$ are closed subsets of  $X_\theta$. Suppose by contradiction that there exist $\overline{m},\overline{n},\overline{r},\overline{s}$ such that $\mathcal C:=C_{\overline{m},\overline{n},\overline{r},\overline{s}}$ has a nonempty interior. Thus there exist $\varepsilon>0$ and $u_0\in \mathcal C$  such that 
\begin{equation}\label{palla}
B_\varepsilon(u_0)\doteq\{v \in X_\theta: \|v - u_0\|_{C^\theta_{x,t}} \le \varepsilon\} \subset \mathcal C.
\end{equation}
By the definition of $X_\theta$, we can find a solution of \eqref{E}, $u \in C^{\theta'}(\T^3\times [0,T])$, $\theta' >\theta$, such that $\|u - u_0\|_{C^{\theta}_{x,t}} \le \frac{\varepsilon}{3}$. Moreover, \eqref{palla} implies that
\begin{equation}\label{palla2}
B_{\frac{\varepsilon}{2}}(u) \subset \mathcal C.
\end{equation}
From now on, we assume that 
\begin{equation}\label{assu}
\theta^* <(\theta')^*< \theta^* + \frac{1}{2\overline{m}}.
\end{equation}
This can be done simply by choosing a possibly smaller $\theta'$ and exploiting the embedding $C^\alpha(\T^3\times [0,T]) \subset C^\beta(\T^3\times [0,T])$, for any $\beta \le \alpha$.
Now fix parameters $\theta'',\beta,\eta > 0$ such that $\theta < \theta'< \theta''< \beta <\eta$ and for which $\eta^* < \theta^* + \frac{1}{2\overline{m}}$. This can be done in view of $\eqref{assu}$. Fix moreover a function (of time only) $f \in C^{\eta^*}([0,T])\setminus W^{\eta^*}$, such that $ 1/2 \leq f\leq 1 $ and set
\begin{equation}\label{form}
e(t)=\int_{\T^3} |u|^2 \, dx +\frac{\rho}{2}f(t),
\end{equation}
for some small parameter $\rho > 0$. These choices imply that the energy $e=e(t)$ satisfies
\begin{equation}\label{en_not_reg}
e\not \in W^{\theta^*+\frac{1}{\bar m},q_{\bar s}}(I_{\bar r}).
\end{equation}
Now we claim that, if $\rho$ is chosen sufficiently small, depending on $\theta, \theta',\theta'',\beta,\eta$ and $\bar m$, then there exists a solution of \eqref{E} $v \in C^{\theta''}(\T^3\times [0,T])$ such that 
\begin{align}
&\|u - v\|_{C^\theta_{x,t}}\le \frac{\varepsilon}{3}, \label{vicinanza}\\
&e_v(t) = e(t),\quad \forall t \in [0,T] \label{EN}.
\end{align}
It is clear that the claim implies a contradiction with \eqref{palla2}. Indeed, since $\theta'' > \theta$, we have $v \in X_\theta$. Therefore, by \eqref{palla2} and \eqref{vicinanza}, we get $e_v \in W^{\theta^*+\frac{1}{\bar m},q_{\bar s}}(I_{\bar r})$, but this is in contradiction with \eqref{EN} and \eqref{en_not_reg}. This would conclude the proof of the present theorem, hence we are only left with the proof of the claim.

\medskip

To prove the claim, we want to apply Proposition \ref{p:main}. First, as in the proof of Theorem \ref{t_exist}, we use the rescaling \eqref{recaling} on $u$ for some  $\Gamma$ to be determined later. In this way, we obtain a new solution  $\tilde u \in C^{\theta'}(\T^3\times [0,T/\Gamma])$. In fact, for any map $w \in C^{\theta'}(\T^3\times [0,T])$, we denote with $\tilde w$ the map obtained through the rescaling \eqref{recaling} with such $\Gamma$. Note that 
\begin{equation}\label{comparable}
\|\tilde w_1 - \tilde w_2\|_{\theta'} = \Gamma \|w_1 - w_2\|_{\theta'},\qquad  \forall w_1,w_2 \in C^{\theta'}(\T^3\times [0,T]),
\end{equation}
and 
\begin{equation}\label{transf}
e_{\tilde w}(t) = \Gamma^2 e_{w}(\Gamma t),  \quad  \forall t \in [0,T/\Gamma].
\end{equation}
Therefore, we also define
\begin{equation}\label{defen}
\tilde e(t) \doteq \Gamma^2 e(\Gamma t),\quad  \forall t \in [0,T/\Gamma].
\end{equation}

However, Proposition \ref{p:main} requires a smooth starting triple. For this reason we consider a space-time mollification  of $\tilde u$, $u_\delta\doteq ( \tilde u * \varphi_\delta) *\psi_\delta$, where $\varphi_\delta$ and $\psi_\delta$ are standard mollifiers in space and time respectively and $\delta > 0$ is a parameter that will be fixed later on.

Note that here we are mollifying in the time variable which belongs to a compact interval. To do that correctly, one should introduce 
\[
	u'(x,s):=\frac{1}{\gamma}\tilde u \left(x,\frac{s+\sigma}{\gamma} \right).
	\]
 This solves the Euler equations on $\T^3\times\left[-\sigma, \gamma \frac{T}{\Gamma} - \sigma \right]$. Furthermore,  by choosing $\sigma > 0$ sufficiently small and $\gamma > 1$ sufficiently close to $1$ we can make $ \|u' - \tilde u \|_\theta$ arbitrarily small. Finally, we can choose $\gamma$ in terms of $\sigma$ in such a way that $[0,T/\Gamma]$ is contained in the interior of $\left[-\sigma, \gamma \frac{T}{\Gamma} - \sigma \right]$. One can now mollify $u'$ on $[0,T/\Gamma]$ instead of $\tilde u$ and proceed with the rest of the proof. For simplicity, we will just keep using $\tilde u$.

  Now the newly introduced map $u_\delta$ is smooth and solves the following Euler-Reynolds system
$$
\partial_t u_\delta +\diver (u_\delta \otimes u_\delta) +\nabla p_\delta=\diver \mathring R_\delta,
$$
where $\mathring R_\delta \doteq u_\delta \mathring \otimes u_\delta - (\tilde u \mathring \otimes \tilde u)_\delta$ and the trace part of the commutator $u_\delta  \otimes u_\delta - (\tilde u  \otimes \tilde u)_\delta$ is inside the pressure $p_\delta$.\\
\\
We now want to take $(u_\delta, p_\delta, R_\delta)$ as a starting point for the iterative scheme given by Proposition \ref{p:main}. In order to do so, we need to guarantee estimates \eqref{e:R_q_inductive_est}, \eqref{e:v_q_inductive_est}, \eqref{e:v_q_0} and to find $\rho >0$ for which also \eqref{e:energy_inductive_assumption} is satisfied with $q = 0$. Recall the definition of $\lambda_q$ and $\delta_q$ of \eqref{e:size_def} and \eqref{e:freq_def}. We set
\[
 \Gamma\doteq \delta_1^{\sfrac{1}{2}}\lambda_1^{\theta +\alpha}\quad \text{ and } \quad \rho \doteq \frac{\delta_1}{\Gamma^2},
\]
and we postpone the choice of $\delta$.\ Note that with this choice, obviously  $\rho$ depend on the parameters appearing in Proposition \ref{p:main}. In particular the energy profile depends on $a$, but this will not be a problem since we will bound $\|e\|_{\eta^*}$ independently of $a$. See also Remark \ref{REM} for a more thorough explanation.\ We start with \eqref{e:v_q_0}. Using \eqref{mollest2} and the rescaling, we get
$$
\|u_\delta\|_0\leq \|\tilde u\|_0 \leq \Gamma \|u\|_0
$$
It is clear that we can find a sufficiently large $a$ such that
\[
\Gamma \|u\|_0 \le 1 - \delta_1^{1/2}.
\]
Therefore, \eqref{e:v_q_0} is fulfilled. Let us now show \eqref{e:R_q_inductive_est} and \eqref{e:v_q_inductive_est}. First, by \eqref{mollest4}, we have
$$
\| \mathring R_\delta \|_0\leq \|\tilde u\|^2_{\theta'}\delta^{2\theta '} = \| u\|^2_{\theta'} \Gamma^2\delta^{2\theta'}.
$$
Thus, if $\delta$ is chosen so that 
\begin{equation}
\label{delta_restr_1}
\delta\leq \left(\frac{\delta_1 \lambda_0^{-4\alpha}}{\Gamma^2}\right)^\frac{1}{2\theta'},
\end{equation}
then \eqref{e:R_q_inductive_est} holds for $q = 0$ if $a$ is large enough. Moreover, through \eqref{mollest2}, we estimate
$$
\|u_\delta\|_1\lesssim  \|\tilde u\|_{\theta'}\delta^{\theta ' -1} = \|u\|_{\theta'}\Gamma \delta^{\theta ' -1}.
$$
Thus, \eqref{e:v_q_inductive_est} holds for $q=0$ as soon as 
\begin{equation}
\label{delta_restr_2}
\delta\geq \left( \frac{\Gamma}{\delta_0^{\sfrac{1}{2}} \lambda_0^{1-\alpha}}\right)^\frac{1}{1-\theta'},
\end{equation}
and $a$ is sufficiently large.  To be able to pick an admissible parameter $\delta$, we need to check the compatibility of \eqref{delta_restr_1} and \eqref{delta_restr_2}. The two are compatible if and only if
\begin{equation}\label{compatib_delta}
\left(\Gamma \delta_0^{-\sfrac{1}{2}} \lambda_0^{-1+\alpha} \right)^{\frac{1}{1-\theta'}} \leq \left( \delta_1 \lambda_0^{-4\alpha}\Gamma^{-2}\right)^\frac{1}{2\theta'}.
\end{equation}
Plugging in the definition of $\Gamma$,  the above inequality holds if 
\begin{equation}\label{ugly_restriction}
\left(  b(\theta+\alpha-\beta) + \beta-1+\alpha\right)(\theta')^*<-2\beta b - 4\alpha + 2b(\beta-\theta-\alpha).
\end{equation}
Since the parameter $\alpha$ can be chosen sufficiently small depending on all the others, it is enough to check that there exists $b>1$ such that \eqref{ugly_restriction} holds for $\alpha=0$.  This is equivalent to 
$$
\frac{(1-\beta)(\theta')^*}{2\theta - (\beta-\theta)(\theta')^*}>1.
$$
Using that $(1-\theta')(\theta')^*=2\theta'$, and since $\beta-\theta>\beta-\theta'$,  the above quantity is bounded from below by 
$$
\frac{(1-\beta)(\theta')^*}{2\theta - (\beta-\theta)(\theta')^*}=\frac{2\theta' - (\beta-\theta')(\theta')^*}{2\theta - (\beta-\theta)(\theta')^*}>\frac{\theta'}{\theta}>1.
$$
This concludes the proof of the compatibility of \eqref{delta_restr_1} and \eqref{delta_restr_2}.

We are left with the estimate on the energy \eqref{e:energy_inductive_assumption}. By using \eqref{mollest4}, we estimate
\begin{align*}
\tilde e(t)-\int_{\T^3} |u_\delta|^2 \, dx&=\int_{\T^3} |\tilde u|^2 \, dx +\frac{\delta_1}{2}f(\Gamma t)-\int_{\T^3} |u_\delta|^2 \, dx=\int_{\T^3}\left( \left(|\tilde u|^2\right)_\delta-|u_\delta|^2 \right) \, dx+\frac{\delta_1}{2}f(\Gamma t)\\
&\leq \|u\|^2_{\theta'} \Gamma^2 \delta^{2\theta '} +\frac{\delta_1}{2}\leq \|u\|^2_{\theta'}\delta_1\lambda_0^{-4\alpha}+\frac{\delta_1}{2},
\end{align*}
where to obtain the last inequality we have used \eqref{delta_restr_1}. If $a$ is large enough,
\[
 \|u\|^2_{\theta'}\delta_1\lambda_0^{-4\alpha}+\frac{\delta_1}{2}\leq \delta_1,
\]
hence the upper bound of \eqref{e:energy_inductive_assumption} holds. Similarly we have
\begin{align*}
\int_{\T^3}\left( \left(|\tilde u|^2\right)_\delta-|u_\delta|^2 \right) \, dx+\frac{\delta_1}{2}f(\Gamma t) \geq-  \|u\|^2_{\theta'} \Gamma^2 \delta^{2\theta '} +\frac{\delta_1}{4} \ge - \|u\|^2_{\theta'}\delta_1\lambda_0^{-4\alpha}+\frac{\delta_1}{4}\geq \delta_1\lambda_0^{-\alpha},
\end{align*}
where, to guarantee the last inequality, we took again the parameter $a$ large enough. Now we observe that, since $\Gamma, \delta_1 \le 1$ for any choice of the parameters,
\[
\|\tilde e\|_{\eta^*} \leq \Gamma^{2+\eta^*} \|e_u\|_{\eta^*} +\frac{\delta_1}{2} \|f\|_{\eta^*}\leq  \|e_u\|_{\eta^*} + \|f\|_{\eta^*}.
\]
Hence, independently of $a$, there exists a constant $E > 0$ such that
\[
\|\tilde e\|_{\eta^*}  \le E, \quad  \forall a \in (0,+\infty).
\]
Therefore, we are in place to apply Proposition \ref{p:main} to get a solution $\tilde v\in C^{\theta ''}(\T^3\times[0,T/\Gamma])$ of \eqref{E}, for  $\theta<\theta '' <\beta$.
Moreover
\begin{equation}\label{rescen}
e_{\tilde v}(t) = \int_{\T^3} |\tilde v|^2\,dx =\tilde e(t)
\end{equation}
and, as already done in \eqref{stimalimite}, we have the estimate
\begin{equation}\label{sigma}
\|\tilde v-u_\delta\|_{\theta }\lesssim \sum_{q\geq 1} \lambda_q^{\theta -\beta}\leq C \lambda_{1}^{\theta-\beta},
\end{equation}
provided $a$ is chosen sufficiently large, with a constant $C>0$ independent of $a$.  Thus,  by triangular inequality 
\begin{equation}\label{piccolo2}
\|\tilde v-\tilde u\|_{\theta }\leq \|\tilde v-u_\delta\|_\theta+\|u_\delta-\tilde u\|_\theta\leq C \lambda_1^{\theta-\beta} + \|u\|_{\theta'}\Gamma \delta^{\theta'-\theta}
\end{equation}
We now rescale $\tilde v$ back to $[0,T]$, i.e. we set $v(x,t):=\Gamma^{-1} \tilde v(x,\Gamma^{-1}t)$. Undo the rescaling also on $\tilde u$, getting back the original $u$ from \eqref{palla2}.  By \eqref{piccolo2},  \eqref{delta_restr_1} and the choice of $\Gamma$ we get
$$
\|v-u\|_\theta=\Gamma^{-1} \|\tilde v - \tilde u\|_\theta \leq C\Gamma^{-1}\lambda_1^{\theta-\beta} +\| u\|_{\theta'} \delta^{\theta'-\theta}\leq C\left(\lambda_1^{-\alpha} + \lambda_0^{-2\alpha\frac{\theta'-\theta}{\theta'}} \right).
$$
Then, Proposition \ref{stimaintempo} implies 
\begin{equation}\label{last_est_baire_close}
\|v-u\|_{C^\theta_{x,t}}\leq  C\left(\lambda_1^{-\alpha} + \lambda_0^{-2\alpha\frac{\theta'-\theta}{\theta'}} \right).
\end{equation}
Thus, if $a$ is sufficiently large, we can guarantee the validity of \eqref{vicinanza}. Moreover, \eqref{EN} follows from \eqref{rescen}, concluding the proof of the claim and of the theorem.

\begin{remark}\label{REM}
Since the choice in the previous proof of the energy profile depends on $a$, we wish to clarify in this remark the dependences of the parameters appearing in the proof of the claim. First, we fixed parameters $0 < \beta < \theta' < 1/3$, and we chose $b>1$ sufficiently close to $1$ in such a way that  \eqref{compatib_delta} and $$b < \sqrt{\frac{\theta'^*}{\beta^*}}$$ hold at the same time.  Note that \eqref{compatib_delta} is important in order to be able to find an admissible mollification parameter $\delta$. By choosing $\alpha \in (0,\alpha_1)$, where $\alpha_1$ is small enough, this can be guaranteed. Note that in this way $\alpha_1$ only depends on $\beta,\theta'$ and $b$, as stated in Proposition \ref{p:main}. Therefore, we can always consider $\alpha_1 \le \alpha_0$, where $\alpha_0$ is the number appearing in Proposition \ref{p:main}. Next, we have proved that there exists $a_1$ large enough such that for $a \geq a_1$, we can guarantee estimates \eqref{e:R_q_inductive_est}, \eqref{e:v_q_inductive_est}, \eqref{e:v_q_0} and \eqref{e:energy_inductive_assumption} for $q = 0$, for any function $e$ of the form \eqref{form}. This $a_1$ only depends on $\beta,b,\alpha,\theta'$ and $u$. Moreover, in the last steps it is required to take $a$ large enough so that the right hand side in \eqref{last_est_baire_close} is less than $\frac{\eps}{3}$. This yields therefore a number $a_2 \ge a_1$ that depends on $\varepsilon$, $E\doteq \|e_u\|_{\eta^*} + \|f\|_{\eta^*}$ and the universal constant $C$ of Proposition \ref{stimaintempo}. Therefore $a_2$ now depends only on $\beta,b,\alpha,\theta'$ and $E$, since $u,\varepsilon$ and $C$ are fixed from the start of the proof of the claim. We can therefore take any $a_2 \ge a_0$, where $a_0$ is the parameter appearing in Proposition \ref{p:main}. Hence we take $\alpha\doteq \frac{\alpha_2}{2}$, $a\doteq 2a_2$. These choices define uniquely $e$ as in \eqref{form} and let us prove the claim.
\end{remark}

We end this section with the proof of Theorem \ref{t_smooth}. Since it follows closely the one of Theorem \ref{t_baire}, we avoid to give all the technical details already given in the previous proof
\subsection{Proof of Theorem \ref{t_smooth}}
We want to prove that $\overline S$ has empty interior, which is equivalent to say that for any $u_0\in \overline S$ and for any $\varepsilon>0$, there exists a $v\in C_\theta$ such that 
\begin{align}
v &\not \in \overline S\label{noncista}, \\
d(u_0,v)&<\varepsilon \label{moltovicino}.
\end{align}
The closure and the distance appearing in the lines above are all referred to the $C^\theta$ topology. Since $u_0\in \overline S$, there exists a smooth solution  $u $ of \eqref{E} such that 
\begin{equation}\label{vicinamezzi}
d(u_0, u)<\frac{\varepsilon}{2}.
\end{equation}
In particular $u$ is a smooth subsolution whose associated Reynolds stress is zero, and by applying the same rescaling \eqref{recaling} with $\Gamma=\min\{(2\|u\|_0)^{-1},1\}$ we can guarantee that the rescaled solution $\tilde u$ satisfies \eqref{e:v_q_inductive_est} and \eqref{e:v_q_0} by choosing the parameter $a$ large enough. Since $\tilde u$ a smooth solution, its kinetic energy is constant, denoted by $E_{\tilde u}$. Moreover, by choosing a non constant and smooth function $1/2\leq f\leq 1 $, with the choice 
$$
e(t)=E_{\tilde u}+\frac{\delta_1}{2}f(t)
$$
also condition \eqref{e:energy_inductive_assumption} is satified. As in the proof of  Theorem \ref{t_baire}, we can now apply Proposition \ref{p:main} in order to get a solution $\tilde v\in C^\theta(\T^3\times [0,T/\Gamma])$, such that $e_{\tilde v}\equiv e$. Moreover, by choosing the parameter $a$ large enough, we can also ensure that $\tilde v$ is $C^\theta$ close to $\tilde u$. By rescaling these maps back, we thus get a solution $v\in C_\theta$ with a non constant energy profile, such that 
\begin{equation}\label{vicinamezzi2}
d(v,u)<\frac{\varepsilon}{2}.
\end{equation}
From \eqref{vicinamezzi} and \eqref{vicinamezzi2} we obviously deduce \eqref{moltovicino}. Moreover, the fact that the kinetic energy $e_v$ is not constant implies that $v$ cannot be obtained as a uniform limit of smooth solutions, showing also \eqref{noncista}.

\section{Proof of Proposition \ref{p:main}}
The proof of the main iterative proposition given in \cite{BDSV} is subdivided in three steps
\begin{itemize}
\item[1.] mollification: $ (v_q,\mathring R_q)\mapsto (v_\ell,\mathring R_\ell)$;
\item[2.]  gluing : $ (v_\ell,\mathring R_\ell) \mapsto (\overline v_q,\mathring{\overline{ R}}_q)$;
\item[3.]  perturbation: $(\overline v_q,\mathring{\overline{ R}}_q)\mapsto (v_{q+1},\mathring R_{q+1}) $.
\end{itemize}
In the proof of \cite[Proposition 2.1]{BDSV}, the energy function $e$ only appears in the perturbation step and both the mollification and the gluing steps are independent on its choice. Thus, also in our case, given the triple  $(v_q,p_q,\mathring R_q)$ there will exists a new triple $(\overline{v}_q, \overline p_q,\mathring{\overline{ R}}_q)$ solving the Euler Reynolds system such that the temporal support of $ \mathring{\overline{ R}}_q$ is contained in pairwise disjoint intervals $I_i$ of length comparable to
$$
\tau_q=\frac{\ell^{2\alpha}}{\delta_q^{\sfrac12}\lambda_q}.
$$
More precisely, for any $n\in \mathbb{Z}$ let 
$$
t_n=\tau_q n, \qquad I_n=\left[t_n+\frac13 \tau_q,t_n+\frac23 \tau_q\right]\cap[0,T], \qquad  J_n=\left[t_n-\frac13 \tau_q,t_n+\frac13 \tau_q\right]\cap[0,T].
$$
We have 
$$
\supp \mathring{\overline R}_q\subset \bigcup_{n\in \mathbb{Z}} I_n \times \T^3.
$$
Moreover  the following estimates hold
\begin{align}
\|v_q-\overline v_q\|_0&\lesssim\delta_{q+1}^{\sfrac12}\lambda_q^{-\alpha} \label{gluing_est_v}\\
\|\overline{v}_q\|_{1+N}&\lesssim \delta_{q}^{\sfrac12}\lambda_q\ell^{-N}\label{gluing_est_v_der}\\
\left\| \mathring{\overline R}_q\right\|_{N+\alpha}&\lesssim \delta_{q+1}\ell^{-N+\alpha}\label{gluing_est_R}\\
\left\| \partial_t \mathring{\overline R}_q+ (\overline{v}_q\cdot \nabla)\mathring{\overline R}_q\right\|_{N+\alpha}&\lesssim \delta_{q+1}\delta_q^{\sfrac12}\lambda_q \ell^{-N-\alpha}\label{gluing_est_adv_R}\\
\left|\int_{\T^3}|\overline v_q|^2 - |v_\ell|^2\,dx\right|&\lesssim \delta_{q+1}\ell^\alpha\label{gluing_est_en_prel},
\end{align}
for any $N\geq 0$, where the small parameter $\ell$ is defined as
$$
\ell=\frac{\delta_{q+1}^{\sfrac12}}{\delta_q^{\sfrac12}\lambda_q^{1+\sfrac{3\alpha}{2}}}
$$
and it comes from the mollification step. We observe that by choosing $\alpha$ sufficiently small and $a$ sufficiently large we can assume
\begin{equation}\label{rel_ell_lambda}
\lambda_q^{-\sfrac32}\leq \ell\leq \lambda_q^{-1}.
\end{equation}
We also state another inequality we will need in the following, that is a consequence of \eqref{mollest4},\eqref{e:energy_inductive_assumption}, and \eqref{gluing_est_en_prel} :
\begin{equation}\label{gluing_est_en}
\frac{\delta_{q+1}}{2\lambda_q^\alpha}\leq e(t)-\int_{\T^3}|\overline v_q|^2\,dx\leq 2\delta_{q+1}.
\end{equation}
\\
Thus we can pass to the perturbation step. The aim is to find a triple $( v_{q+1},p_{q+1},\mathring R_q)$ which solves \eqref{NSR} with the estimates
\begin{align}
\|v_{q+1}-\overline v_q\|_0+\lambda_{q+1}^{-1}\|v_{q+1}-\overline v_q\|_1&\leq\frac{M}{2}\delta_{q+1}^{\sfrac12} \label{pert_est_v}\\
\left\| \mathring{ R}_{q+1}\right\|_{\alpha}&\lesssim \frac{\delta^{\sfrac12}_{q+1}\delta_q^{\sfrac12}\lambda_q}{\lambda_{q+1}^{1-4\alpha}}\label{pert_est_R}\\
\Bigg| e(t)-\int_{\T^3}|v_{q+1}|^2\,dx -\frac{\delta_{q+2}}{2}\Bigg| &\leq C\frac{\delta_q^{\sfrac12}\delta^{\sfrac12}_{q+1}\lambda_q^{1+2\alpha}}{\lambda_{q+1}}+\frac{\delta_{q+2}}{4}\label{pert_est_en}.
\end{align}
Note that estimates \eqref{pert_est_v} and \eqref{pert_est_R} are the same stated in \cite{BDSV}, while \eqref{pert_est_en} is slightly different due to the term $\delta_{q+2}/4$. This does not affect the iteration and Proposition \ref{p:main} is still a direct consequence of estimates \eqref{pert_est_v}-\eqref{pert_est_en}. However, since estimate \eqref{pert_est_en} is different than the one used in \cite{BDSV}, we give a complete proof of Proposition \ref{p:main}.
\subsection{Proof of Proposition \ref{p:main}}
By using \eqref{gluing_est_v} and \eqref{pert_est_v} we estimate
$$
\|v_{q+1}-v_q\|_0\leq \|v_{q+1}-\overline v_q\|_0+\|\overline v_{q}-v_q\|_0\leq \frac{M}{2}\delta_{q+1}^{\sfrac12}+C \delta_{q+1}^{\sfrac12} \lambda_q^{-\alpha},
$$
where the constant $C$ depends only on $\alpha, \beta$ and $M$. Thus if $a$ is chosen sufficiently large we can guarantee 
\begin{equation}\label{est_v_0}
\|v_{q+1}-v_q\|_0\leq M\delta_{q+1}^{\sfrac12}.
\end{equation}
Similarly, by using \eqref{e:v_q_inductive_est}, \eqref{gluing_est_v_der} and \eqref{pert_est_v}, we have
\begin{align*}
\|v_{q+1}-v_q\|_1\leq \|v_{q+1}-\overline v_q\|_1+\|\overline v_{q}\|_1+\|v_q\|_1\leq \frac{M}{2}\delta_{q+1}^{\sfrac12}\lambda_{q+1}+\left(C+M\right)\delta_q^{\sfrac12}\lambda_q.
\end{align*}
Again, if $a$ is chosen sufficiently large, we can ensure 
\begin{equation*}
\|v_{q+1}-v_q\|_1\leq M\delta_{q+1}^{\sfrac12}\lambda_{q+1},
\end{equation*}
which, together with \eqref{est_v_0}, gives \eqref{e:v_diff_prop_est}. 
By \eqref{e:v_q_inductive_est}, \eqref{e:v_q_0} and \eqref{e:v_diff_prop_est} we get 
\begin{align*}
\|v_{q+1}\|_0&\leq \|v_{q+1}-v_q\|_0+\|v_q\|_0\leq \frac{M}{2}\delta_{q+1}^{\sfrac12}+1-\delta_q^{\sfrac12}\leq 1-\delta_{q+1}^{\sfrac12},\\
\|v_{q+1}\|_1&\leq \|v_{q+1}-v_q\|_1+\|v_q\|_1\leq  \frac{M}{2}\delta_{q+1}^{\sfrac12}\lambda_{q+1}+ M\delta_q^{\sfrac12}\lambda_q\leq M \delta_{q+1}^{\sfrac12}\lambda_{q+1}
\end{align*}
where we also chose the parameter $a$ sufficiently large to guarantee the last inequalities of the previous estimates. In particular this shows that $v_{q+1}$ obeys \eqref{e:v_q_inductive_est} and \eqref{e:v_q_0} in which $q$ is replaced by $q+1$.
Estimate \eqref{e:R_q_inductive_est} for $\mathring{R}_{q+1}$ is a direct consequence of \eqref{pert_est_R}  and the parameters inequality 
\begin{equation}\label{paramet_ineq}
 \frac{\delta^{\sfrac12}_{q+1}\delta_q^{\sfrac12}\lambda_q}{\lambda_{q+1}}\leq \frac{\delta_{q+2}}{\lambda_{q+1}^{8\alpha}}.
\end{equation}
Indeed, by taking the logarithms, the last inequality holds by choosing $a$ sufficiently large if 
$$
-\beta-\beta b+1-b +2b^2\beta+8b\alpha<0,
$$
but this is true since $b < \frac{1-\beta}{2\beta}$ (see \eqref{b_comecamillo}) and $\alpha$ is chosen sufficiently small. We are only left with estimate \eqref{e:energy_inductive_assumption}  for $v_{q+1}$. By \eqref{pert_est_en} and \eqref{paramet_ineq} we have
$$
e(t)-\int_{\T^3}|v_{q+1}|^2\,dx \leq \frac{\delta_{q+2}}{2}+C\frac{\delta_q^{\sfrac12}\delta^{\sfrac12}_{q+1}\lambda_q^{1+2\alpha}}{\lambda_{q+1}}+\frac{\delta_{q+2}}{4}\leq \frac{3}{4}\delta_{q+2}+C \frac{\delta_{q+2}}{\lambda_{q+1}^{6\alpha}},
$$
thus, for a sufficiently large $a$, we get
\begin{equation}\label{est_energia}
e(t)-\int_{\T^3}|v_{q+1}|^2\,dx \leq \delta_{q+2}.
\end{equation}
Finally, again by \eqref{pert_est_en} we have
$$
e(t)-\int_{\T^3}|v_{q+1}|^2\,dx \geq  \frac{\delta_{q+2}}{2}-C\frac{\delta_q^{\sfrac12}\delta^{\sfrac12}_{q+1}\lambda_q^{1+2\alpha}}{\lambda_{q+1}}-\frac{\delta_{q+2}}{4}\geq \left(\frac{1}{4} -\frac{C}{\lambda_{q+1}^{6\alpha}}\right) \delta_{q+2},
$$
and, since  for a sufficiently large $a$ we can ensure that 
$$
\frac{1}{4}-\frac{C}{\lambda_{q+1}^{6\alpha}}\geq \frac{1}{\lambda_{q+1}^{\alpha}},
$$
we end up with
$$
e(t)-\int_{\T^3}|v_{q+1}|^2\,dx \geq\delta_{q+2}\lambda_{q+1}^{-\alpha},
$$
which together with \eqref{est_energia} gives \eqref{e:energy_inductive_assumption} and concludes the proof of the proposition.
\section{Perturbation}

We will now outline the construction of the perturbation $w_{q+1}$, where 
\[
v_{q+1}:= w_{q+1} + \overline v_q \, .
\] 
The perturbation $w_{q+1}$ is highly oscillatory and will be based on the Mikado flows introduced in \cite{DS}. We recall the construction in the following lemma

\begin{lemma}\label{l:Mikado}For any compact subset $\mathcal N\subset\subset \mathcal{S}^{3\times3}_+$
there exists a smooth vector field 
$$
W:\mathcal N\times \T^3 \to \R^3, 
$$
such that, for every $R\in\mathcal N$ 
\begin{equation}\label{e:Mikado}
\left\{\begin{aligned}
\diver_\xi(W(R,\xi)\otimes W(R,\xi))&=0 \\ \\
\diver_\xi W(R,\xi)&=0,
\end{aligned}\right.
\end{equation}
and
\begin{eqnarray}
	\fint_{\T^3} W(R,\xi)\,d\xi&=&0,\label{e:MikadoW}\\
    \fint_{\T^3} W(R,\xi)\otimes W(R,\xi)\,d\xi&=&R.\label{e:MikadoWW}
\end{eqnarray}
\end{lemma}

Using the fact that $W(R,\xi)$ is $\T^3$-periodic and has zero mean in $\xi$, we write
\begin{equation}\label{e:Mikado_Fourier}
W(R,\xi)=\sum_{k \in \Z^3\setminus\{0\}}  a_k(R) e^{ik\cdot \xi}
\end{equation}
for some  smooth functions $R\to a_k(R) \in \C^3$, satisfying $a_k (R) \cdot k=0$. From the smoothness of $W$, we further infer
\begin{equation}\label{e:a_k_est}
\sup_{R\in \mathcal N}|D^N_R a_k(R)|\leq  \frac{C(\mathcal{N},N,m)}{|k|^m}
\end{equation}
for some constant $C$, which depends, as highlighted in the statement, on $\mathcal{N}$, $N$ and $m$.
\begin{remark}\label{r:choice_of_M}
Later in the proof the estimates \eqref{e:a_k_est} will be used with a specific choice of the compact set $\mathcal{N}$ and of the integers $N$ and $m$: this specific choice will then determine the universal constant $M$ appearing in Proposition \ref{p:main}.
\end{remark}

Using the Fourier representation we see that from \eqref{e:MikadoWW}
\begin{equation}\label{e:Mikado_stationarity}
W(R,\xi)\otimes W(R,\xi) = R+\sum_{k\neq 0} C_{k}(R) e^{i k\cdot \xi}
\end{equation}
where
\begin{equation}\label{e:Ck_ind}
C_k  k=0 \quad \mbox{and} \quad
\sup_{R\in \mathcal N}|D^N_R C_k(R)|\leq \frac{C (\mathcal{N}, N, m)}{|k|^m}
\end{equation}
for any $m,N \in \N$. It will also be useful to write the Mikado flows in terms of a potential. We note
\begin{align}
\curl_{\xi}\left(\left(\frac{ik\times  a_k}{|k|^2}\right) e^{i k\cdot \xi}\right) &= -i\left(\frac{ik\times  a_k}{|k|^2}\right)\times k  e^{i k\cdot \xi} 
= -\frac{k\times (k\times  a_k)}{|k|^2}  e^{i k\cdot \xi} =  a_k  e^{i k\cdot \xi} \label{e:Mikado_Potential}
\end{align}

We define the smooth non-negative cut-off functions $\eta_i=\eta_i(x,t)$ with the following properties
\begin{enumerate}
\item[(i)] $\eta_i\in C^{\infty}(\T^3\times [0,T])$ with $0\leq \eta_i(x,t)\leq 1$ for all $(x,t)$;
\item[(ii)] $\supp \eta_i\cap\supp\eta_j=\emptyset$ for $i\neq j$;
\item[(iii)] $\T^3\times I_i\subset \{(x,t):\eta_i(x,t)=1\}$;
\item[(iv)] $\supp \eta_i\subset \T^3\times I_i\cup J_i\cup J_{i+1}$;
\item[(v)] There exists a positive geometric constant $c_0>0$ such that for any $t\in[0,T]$
\begin{equation}\label{c0_constant}
\sum_i\int_{\T^3}\eta_i^2(x,t)\,dx\geq c_0.
\end{equation}
\end{enumerate}

The next lemma is taken from \cite{BDSV}.
\begin{lemma}\label{l:cutoffs}
There exists cut-off functions $\{\eta_i\}_i$ with the properties (i)-(v) above and such that for any $i$ and $n,m\geq 0$
\begin{align*}
\|\partial_t^n\eta_i\|_{m}\leq C (n,m) \tau_q^{-n}
\end{align*}
where $C(n,m)$ are geometric constants depending only upon $m$ and $n$.
\end{lemma}

Analogously to \cite{BDSV}, we will now define the perturbations that are necessary to show \eqref{pert_est_v}-\eqref{pert_est_en}. Since the energy profile is not smooth, we will need to mollify it. To do so we will henceforth consider $e$ to be extended on the whole $\R$ as $e(t)=e(0)$ for all $t<0$ and $e(t)=e(T)$ for all $t>T$, in such a way that the extension is still in $C^{\eta^*}(\R)$. With this convention we define
$$
e_q(t):=(e *\psi_{\varepsilon_q})(t),
$$
where $\psi_{\varepsilon_q}$ is a standard mollifier and
\begin{equation}\label{eps_q}
\varepsilon_q:=\left(\frac{\delta_{q+2}}{4E}\right)^\frac{1}{\eta^*}.
\end{equation}
Define also
\begin{equation*}
\rho_{q}(t):= \frac{1}{3} \left(e_q(t)-\frac{\delta_{q+2}}{2}-\int_{\T^3} |\overline v_q|^2\,dx\right)
\end{equation*}
and
\begin{equation*}
\rho_{q,i}(x,t):= \frac{\eta_i^2(x,t)}{\sum_j \int_{\T^3} \eta_j^2(y,t)\,dy}\rho_{q}(t)
\end{equation*}

Define the backward flows $\Phi_i$ for the velocity field $\overline v_{q}$ as the solution of the transport equation
\begin{equation*}
\left\{ 
\begin{aligned}
&(\partial_t + \overline v_q  \cdot \nabla) \Phi_i =0 \\ \\
&\Phi_i\left(x,t_i\right) = x.
\end{aligned}
\right.
\end{equation*}
Define
\begin{equation*}
R_{q,i}:=\rho_{q,i} \Id- \eta_i^2\mathring{\overline R_q}
\end{equation*}
and
\begin{equation}\label{e:tildeR_def}
\tilde R_{q,i} =  \frac{\nabla\Phi_iR_{q,i}(\nabla\Phi_i)^T}{ \rho_{q,i}} \,.
\end{equation}
We note that, because of properties (ii)-(iv) of $\eta_i$, 
\begin{itemize}
\item $\supp R_{q,i}\subset \supp\eta_i$;
\item on $\supp\mathring{\bar{R}}_q$ we have $\sum_i\eta_i^2=1$;
\item $\supp \tilde R_{q,i}\subset \T^3\times I_i\cup J_i\cup J_{i+1}$;
\item $\supp \tilde R_{q,i}\cap \supp \tilde R_{q,j}=\emptyset\textrm{ for all }i\neq j$.
\end{itemize}

\begin{lemma}
\label{l:R_in_range}
For $a\gg 1$ sufficiently large we have
\begin{equation}\label{e:Phi-close-to-id}
\|\nabla \Phi_i - \Id\|_0 \leq \frac{1}{2} \qquad \mbox{for $t\in \supp (\eta_i)$.}
\end{equation}
Furthermore, for any $N\geq 0$ 
\begin{align}
\frac{\delta_{q+1}}{8\lambda_q^{\alpha}} \leq |\rho_{q}(t)| &\leq \delta_{q+1}\quad\textrm{ for all $t$}\,,
\label{e:rho_range}\\
\|\rho_{q,i}\|_0 &\leq \frac{\delta_{q+1}}{c_0}\,,\label{e:rho_i_bnd}\\
 \|\rho_{q,i}\|_N&\lesssim \delta_{q+1}\,,\label{e:rho_i_bnd_N}\\
 \|\partial_t \rho_q\|_0 &\lesssim \delta_{q+1} \delta_q^{\sfrac{1}{2}} \lambda_q
 \label{e:rho_t}\,,\\
 \|\partial_t \rho_{q,i}\|_N &\lesssim \delta_{q+1}\tau_q^{-1}\,.
 \label{e:rho_i_bnd_t}
\end{align}
Moreover, for all $(x,t)$
$$
\tilde R_{q,i}(x,t)\in B_{\sfrac12}(\Id)\subset \mathcal{S}^{3\times 3}_+\,,
$$
where $B_{\sfrac12}(\Id)$ denotes the metric ball of radius $1/2$ around the identity $\Id$ in the space $\mathcal{S}^{3\times 3}$. 
\end{lemma}
\begin{proof}
We write
$$
\rho_q(t)=\frac{1}{3}\left(e_q(t)-\int_{\T^3}|\overline v_q|^2\,dx -\frac{\delta_{q+2}}{2}\right)=\frac{1}{3}\left(e_q(t)-e(t)+e(t)-\int_{\T^3}|\overline v_q|^2\,dx -\frac{\delta_{q+2}}{2} \right),
$$
thus by \eqref{gluing_est_en} we get
\begin{equation}\label{est_rho_q_prima}
\frac{1}{3}\left( \frac{\delta_{q+1}}{2\lambda_q^\alpha}-\frac{\delta_{q+2}}{2}-|e_q(t)-e(t)|\right)\leq |	\rho_q(t)|\leq \frac{1}{3}\left(|e_q(t)-e(t)|+2\delta_{q+1}+\frac{\delta_{q+2}}{2} \right).
\end{equation}
By using \eqref{mollest2} and the fact that $[e]_{\eta^*} \le E$, we also get
$$
|e_q(t)-e(t)|\leq [e]_{\eta^*}\varepsilon_q^{\eta^*}\leq \delta_{q+2}
$$
and, by plugging it into \eqref{est_rho_q_prima}, we achieve
$$
 \frac{\delta_{q+1}}{6\lambda_q^\alpha}-\frac{ \delta_{q+2}}{2} \leq|	\rho_q(t)|\leq \frac23 \delta_{q+1}+\frac{\delta_{q+2}}{2}.
$$
It is easy to show that by choosing $a$ sufficiently large we can guarantee \eqref{e:rho_range}. Note that by definition of the cut-off function $\eta_i$ 
\begin{equation}\label{bound_cutoff}
c_0\leq \sum_{i} \int_{\T^3} \eta_i^2(x,t)\,dx \leq 2
\end{equation}
and hence we obtain \eqref{e:rho_i_bnd}. Since $|\nabla^N \eta_j|\lesssim 1 $, the bound \eqref{e:rho_i_bnd_N} also follows.
For the bound \eqref{e:Phi-close-to-id} and the fact that $ \tilde R_{q,i}(x,t)\in B_{\sfrac12}(\Id)\subset \mathcal{S}^{3\times 3}_+ $ we refer to \cite[Lemma 5.4]{BDSV}.
To prove \eqref{e:rho_t}, we first use \eqref{gluing_est_v_der}, \eqref{gluing_est_R} to estimate
$$
\left| \frac{d}{dt}\int_{\T^3} |\overline v_q|^2\,dx \right| =2\left|\int_{\T^3} \nabla \overline v_q \cdot \mathring{\overline R}_q\,dx \right| \lesssim \delta_{q+1}\delta_q^{\sfrac12}\lambda_q.
$$
Moreover, by \eqref{mollest3}, we have 
$$
|\partial_te_q|\leq [e]_{\eta^*} \varepsilon_q^{\eta^* - 1} \leq C \delta_{q+2}^{1-\sfrac{1}{\eta^*}},
$$
where the constant $C$ depends on $\eta$ and $E$. Thus \eqref{e:rho_t} is implied by the following parameters inequality 
\begin{equation}\label{relazionestrana}
C \delta_{q+2}^{1-\sfrac{1}{\eta^*}}\leq  \delta_{q+1}\delta_q^{\sfrac12}\lambda_q.
\end{equation}
Using the definition of the parameters $\delta_q$ and $\lambda_q$ it can be checked that the last inequality holds if one chose $a$ big enough (depending on $b, \beta, \eta$ and $E$) provided that 
$$
\left( \frac{1}{\eta^*}-1\right) b^2+b-\frac{1}{\beta^*}<0.
$$
Since $b$ satisfies \eqref{e:b_beta_rel} we have
$$
\left( \frac{1}{\eta^*}-1\right) b^2+b-\frac{1}{\beta^*}<\left( \frac{1}{\eta^*}-1\right)\frac{\eta^*}{\beta^*}+\frac{\eta^*}{\beta^*}-\frac{1}{\beta^*}=0,
$$
thus \eqref{relazionestrana} holds. Finally, since $\|\partial_t \eta_j \|_N \lesssim \tau_q^{-1}$ and $\tau_q^{-1}\geq \delta_q^{\sfrac12}\lambda_q$, using \eqref{bound_cutoff}, also the estimate \eqref{e:rho_i_bnd_t} follows.
\end{proof}

\subsection{The constant \texorpdfstring{$M$}{M}}
The principal term of the perturbation can be written as
\begin{equation}
w_{o}:=\sum_i  \left(\rho_{q,i}(x,t)\right)^{\sfrac12} (\nabla\Phi_i)^{-1} W(\tilde R_{q,i}, \lambda_{q+1}\Phi_i) = \sum_i w_{o,i}\, ,
\label{e:w0_decomp}
\end{equation}
where Lemma \ref{l:Mikado} is applied with $\mathcal{N} = \overline{B}_{\sfrac12} (\Id)$, namely the closed ball (in the space of symmetric $3\times 3$ matrices) of radius $\sfrac{1}{2}$ centered at the identity matrix.

From Lemma \ref{l:R_in_range} it follows that $W(\tilde R_{q,i}, \lambda_{q+1}\Phi_i)$ is well defined. Using the Fourier series representation of the Mikado flows \eqref{e:Mikado_Fourier} we can write 
\begin{equation*}
w_{o,i}=\sum_{k \neq 0} (\nabla\Phi_i)^{-1} b_{i,k} e^{i\lambda_{q+1}k\cdot \Phi_i}\, ,
\end{equation*}
where 
\begin{equation*}
b_{i,k}(x,t):= \left(\rho_{q,i}(x,t)\right)^{\sfrac12} a_k(\tilde R_{q,i}(x,t)).
\end{equation*}
By the definition of $w_{o,i}$ and \eqref{e:MikadoWW} we compute
\begin{align}\label{ugly_ww}
w_{o,i}\otimes w_{o,i}&=\rho_{q,i} \nabla\Phi_i^{-1}(W\otimes W)(\tilde R_{q,i},\lambda_{q+1}\Phi_i)\nabla \Phi^{-T}_i\nonumber \\
&=\rho_{q,i} \nabla\Phi_i^{-1}\tilde R_{q,i}\nabla \Phi^{-T}_i+\sum_{k\neq 0}\rho_{q,i} \nabla\Phi_i^{-1}C_k(\tilde R_{q,i})\nabla \Phi^{-T}_ie^{i\lambda_{q+1}k\cdot \Phi_i}\nonumber \\
&=R_{q,i}+\sum_{k\neq 0}\rho_{q,i} \nabla\Phi_i^{-1}C_k(\tilde R_{q,i})\nabla \Phi^{-T}_ie^{i\lambda_{q+1}k\cdot \Phi_i}.
\end{align}
The following is a crucial point of the construction, which ensures that the constant $M$ of Proposition \ref{p:main}
is geometric and in particular independent of all the parameters of the construction.

\begin{lemma}\label{l:choice_of_M}
There is a geometric constant $\bar M$ such that
\begin{equation}\label{e:barM}
\|b_{i,k}\|_0 \leq \frac{\bar M}{|k|^{ 4}} \delta_{q+1}^{\sfrac{1}{2}}\, .
\end{equation}
\end{lemma}

We are finally ready to define the constant $M$ of Proposition \ref{p:main}: from Lemma \ref{l:choice_of_M} it follows trivially that the constant is indeed geometric and hence independent of all the parameters of the statement of Proposition \ref{p:main}.

We can now define the geometric constant $M$ as 
\begin{equation}\label{d:choice_of_M}
M = 64 \bar M \sum_{k\in \Z^3\setminus \{0\}} \frac{1}{|k|^4}\, ,
\end{equation}
where $\bar M$ is the constant of Lemma \ref{l:choice_of_M}.

We also define
\begin{align*}
w_{c}&:=\frac{-i}{\lambda_{q+1}}\sum_{i,k\neq 0} \left[ \curl \left(\left(\rho_{q,i}\right)^{\sfrac12} \frac{\nabla\Phi_i^T(k\times a_{k}(\tilde R_{q,i}))}{|k|^2}\right)\right] e^{i\lambda_{q+1}k\cdot \Phi_i}
=: \sum_{i,k\neq 0} c_{i,k} e^{i\lambda_{q+1}k\cdot \Phi_i}\, .
\end{align*}
Then by direct computations one can check that
\begin{align}\label{e:w_curl}
w_{q+1} = w_o+w_c=\frac{-1}{\lambda_{q+1}}\curl\left(\sum_{i,k\neq 0} (\nabla\Phi_i)^T\left(\frac{ik\times b_{k,i}}{|k|^2}\right)e^{i\lambda_{q+1}k\cdot \Phi_i}\right)\,,
\end{align}
thus the perturbation $w_{q+1}$ is divergence free. 

\subsection{The final Reynolds stress and conclusions}  
In order to define the new Reynolds tensor, we recall the operator $\mathcal R$ from \cite{BDSV}, which
can be thought of as an inverse divergence operator for symmetric tracefree 2-tensors. The operator is defined as
\begin{equation}
\label{e:R:def}
\begin{split}
({\mathcal R} f)^{ij} &= {\mathcal R}^{ijk} f^k \\
{\mathcal R}^{ijk} &= - \frac 12 \Delta^{-2} \partial_i \partial_j \partial_k - \frac 12 \Delta^{-1} \partial_k \delta_{ij} +  \Delta^{-1} \partial_i \delta_{jk} +  \Delta^{-1} \partial_j \delta_{ik}.
\end{split}
\end{equation}
when acting on vectors $f$ with zero mean on $\T^3$, and has the property that $\mathcal R f$ is symmetric and $\diver ( {\mathcal R}  f) = f$.  Upon letting
\begin{align*}
\overline R_q = \sum_{i} R_{q,i}\, ,
\end{align*}
we define the new Reynolds stress as follows
\begin{align}\label{new_reynolds}
\mathring{R}_{q+1} :=  \RR \left( w_{q+1} \cdot \nabla \overline v_q+\partial_t  w_{q+1} + \overline v_q \cdot \nabla w_{q+1}+\diver \left(- {\overline R}_{q} + w_{q+1} \otimes w_{q+1} \right)\right) 
\end{align}

With this definition one may verify that 
\begin{equation*}
\left\{
\begin{array}{l}
 \partial_t v_{q+1} + \diver (v_{q+1} \otimes v_{q+1}) + \nabla p_{q+1} = \diver(\mathring{R}_{q+1}) \, ,
\\ \\
 \diver v_{q+1} = 0 \, ,
\end{array}\right.
\end{equation*}
where the new pressure is defined by
\begin{equation}\label{e:new_pressure}
p_{q+1}(x,t) = \bar p_q(x,t)  - \sum_{i} \rho_{q,i}(x,t)  + \rho_{q}(t).
\end{equation}

The following proposition is taken from  \cite{BDSV}.
\begin{prop}
\label{p:perturbation}
For $t\in I_i\cup J_i \cup J_{i+1}$ and any $N\geq 0$
\begin{align}
\| (\nabla\Phi_i)^{-1}\|_N + \|\nabla\Phi_i\|_N &\lesssim \ell^{-N} \,,\label{e:phi_N}\\
\|\tilde R_{q,i}\|_N &\lesssim  \ell^{-N}\,,\label{e:tR_est}\\
\|b_{i,k}\|_N &\lesssim \delta_{q+1}^{\sfrac12}|k|^{-6}\ell^{-N}\,, \label{e:b_k_est_N}\\
\|c_{i,k}\|_N &\lesssim  \delta_{q+1}^{\sfrac12}\lambda_{q+1}^{-1}|k|^{-6}\ell^{-N-1}\,.\label{e:c_k_est}
\end{align}
Moreover assuming $a$ is sufficiently large, the perturbations $w_o$, $w_c$ and $w_q$ satisfy the following estimates
\begin{align}
\|w_o\|_0 +\frac{1}{\lambda_{q+1}}\|w_o\|_1 &\leq \frac{M}{4}\delta_{q+1}^{\sfrac 12}\label{e:w_o_est}\\
\|w_c\|_0+\frac{1}{\lambda_{q+1}} \|w_c\|_1 &\lesssim \delta_{q+1}^{\sfrac 12}\ell^{-1}\lambda_{q+1}^{-1}\label{e:w_c_est}\\
\|w_{q+1}\|_0 +\frac{1}{\lambda_{q+1}}\|w_{q+1}\|_1 &\leq  \frac{M}{2} \delta_{q+1}^{\sfrac 12}\label{e:w_est}
\end{align}
where the constant $M$ depends solely on the constant $c_0$ in \eqref{c0_constant}.
In particular, we obtain \eqref{pert_est_v}.
\end{prop}


We are now ready to complete the proof of Proposition \ref{p:main} by proving the remaining estimates \eqref{pert_est_en} and \eqref{pert_est_R}.
We start with the energy increment
 
\begin{prop}\label{p:energy}
 The energy of $v_{q+1}$ satisfies the following estimate
\begin{equation*}
    \Bigg|e(t)-\int_{\T^3}|v_{q+1}|^2\,dx-\frac{\delta_{q+2}}2 \Bigg|\leq C \frac{\delta_q^{\sfrac12}\delta_{q+1}^{\sfrac12}\lambda_q^{1+2\alpha}}{\lambda_{q+1}}+\frac{\delta_{q+2}}{4}\,.
\end{equation*}
In particular, \eqref{pert_est_en} holds.
\end{prop}
\begin{proof}
By definition we have $v_{q+1}=\overline v_q+w_{q+1}=\overline v_q+w_o+w_c$, thus we have
\begin{align}\label{primodacamillo}
\left|e(t)-\int_{\T^3} |v_{q+1}|^2\,dx-\frac{\delta_{q+2}}{2}  \right| &\leq \left|e(t)-\int_{\T^3} |w_{o}|^2\,dx-\frac{\delta_{q+2}}{2} -\int_{\T^3}|\overline{v}_q|^2\,dx \right|\nonumber \\
&+\left| \int_{\T^3}|w_c|^2\,dx+2\int_{\T^3} w_o\cdot w_c\,dx+2\int_{\T^3}w_{q+1}\cdot \overline v_q\,dx\right|.
\end{align}
The estimate on the second term in the right hand side of \eqref{primodacamillo} is just a a consequence of \eqref{gluing_est_v_der} and Proposition \ref{p:perturbation} and for a complete we refer to \cite[Proposition 6.2]{BDSV}, in which it is proved that
$$
\left| \int_{\T^3}|w_c|^2\,dx+2\int_{\T^3} w_o\cdot w_c\,dx+2\int_{\T^3}w_{q+1}\cdot \overline v_q\,dx\right|\lesssim \frac{\delta_q^{\sfrac12}\delta_{q+1}^{\sfrac12}\lambda_q^{1+2\alpha}}{\lambda_{q+1}}.
$$
Now recall that from \eqref{ugly_ww} and the definition of $R_{q,i}$ we have
\begin{align*}
\int_{\T^3}|w_o|^2\,dx&=\sum_i\int_{\T^3} \tr R_{q,i}\,dx+\int_{\T^3}\sum_{i,k\neq 0} \rho_{q,i} \nabla\Phi_i^{-1}\tr C_k(\tilde R_{q,i})\nabla \Phi^{-T}_ie^{i\lambda_{q+1}k\cdot \Phi_i}\,dx\\
&=3\sum_{i}\int_{\T^3}\rho_{q,i}\,dx +\int_{\T^3}\sum_{i,k\neq 0} \rho_{q,i} \nabla\Phi_i^{-1}\tr C_k(\tilde R_{q,i})\nabla \Phi^{-T}_ie^{i\lambda_{q+1}k\cdot \Phi_i}\,dx\\
&=3\rho_q(t)+\int_{\T^3}\sum_{i,k\neq 0} \rho_{q,i} \nabla\Phi_i^{-1}\tr C_k(\tilde R_{q,i})\nabla \Phi^{-T}_ie^{i\lambda_{q+1}k\cdot \Phi_i}\,dx\\
&=e_q(t)-\frac{\delta_{q+2}}{2}-\int_{\T^3}|\overline v_q|^2\,dx+\int_{\T^3}\sum_{i,k\neq 0} \rho_{q,i} \nabla\Phi_i^{-1}\tr C_k(\tilde R_{q,i})\nabla \Phi^{-T}_ie^{i\lambda_{q+1}k\cdot \Phi_i}\,dx.
\end{align*}
As a consequence of \eqref{e:Ck_ind}, Lemma \ref{l:R_in_range} and Proposition \ref{p:perturbation} we have
$$
\left|\int_{\T^3}\sum_{i,k\neq 0} \rho_{q,i} \nabla\Phi_i^{-1}\tr C_k(\tilde R_{q,i})\nabla \Phi^{-T}_ie^{i\lambda_{q+1}k\cdot \Phi_i}\,dx\right|\lesssim  \frac{\delta_q^{\sfrac12}\delta_{q+1}^{\sfrac12}\lambda_q^{1+2\alpha}}{\lambda_{q+1}}.
$$
For a detailed proof of the previous estimate we again refer to \cite[Proposition 6.2]{BDSV}. Thus we are only left with estimating $|e(t)-e_q(t)| $, but from \eqref{mollest2}, the definition of $\varepsilon_q$ in \eqref{eps_q} and the fact that $[e]_{C^{\eta^*}} \le E$, we get
$$
|e(t)-e_q(t)|\leq [e]_{\eta^*}\varepsilon_q^{\eta^*} \le \frac{\delta_{q+2}}{4},
$$
which concludes the proof of the proposition. 
\end{proof}

For the inductive estimate on $\mathring R_{q+1}$ we refer to  \cite[Proposition 6.1]{BDSV}
\begin{prop}
\label{p:R_q+1}
The Reynolds stress error $\mathring R_{q+1}$ defined in \eqref{new_reynolds} satisfies the estimate
\begin{equation}\label{e:final_R_est}
\|\mathring R_{q+1}\|_{0}\lesssim \frac{\delta_{q+1}^{\sfrac12}\delta_q^{\sfrac{1}{2}} \lambda_q}{\lambda_{q+1}^{1-4\alpha}} \,.
\end{equation}
In particular, \eqref{pert_est_R} holds.
\end{prop}

\section{Final Comments}\label{comm}

In this section, we wish to comment on why we need to introduce the space $X_{\theta}$ (see \eqref{Xtheta}), since clearly the most natural choice for $X_{\theta}$ would have simply been the space of all $C^\theta(\T^3\times [0,T])$ or $c^\theta(\T^3\times [0,T])$ solutions of Euler equation. Here $c^\theta$ denotes the space of little H\"older continuous functions, namely the closure of smooth functions in the $C^\theta$ norm.  We believe that such a discussion highlights some interesting features of the convex integration scheme.
\\
\\
The introduction of $X_{\theta}$ is related to the proof of Theorem \ref{t_baire} and to intrinsic properties of the iterative scheme of \cite{BDSV}. The proof of Theorem \ref{t_baire} uses the following strategy, that is quite standard in arguments involving Baire Theorem. As a first step, we rewrite $Y_\theta^c$ as union of closed sets $C_{m,n,r,s}$. The parameters $m,n,s$ quantify an improvement in the regularity of elements of $C_{m,n,r,s}$. Secondly, one needs to prove that $C_{m,n,r,s}$ has empty interior. Equivalently, every element $u_0 \in C_{m,n,r,s}$ must be approximated in the $C^\theta(\T^3\times[0,T])$ norm with elements $u \in X_{\theta}\setminus C_{m,n,r,s}$. This is where the convex integration scheme comes into play. The iterative procedure of \cite{BDSV} tells us, roughly speaking, that given a smooth subsolution $\bar{u}$ and a positive and smooth (or $C^{\theta^* + \gamma}([0,T])$, as proved in the present work) energy profile $e$, one can find an arbitrarily close solution $u$ such that $e = e_u$, provided some initial estimates are verified. In order to obtain the desired "less regular" approximating sequence, it seems therefore rather natural to try to apply this result to the subsolution obtained by mollifying $u_0$, and choose an energy profile $e \in C^{\theta^* + 1/2m}([0,T])\setminus W^{\theta^* + 1/2m}$.
\\
\\
Since one wishes to approximate a $C^{\theta}(\T^3\times [0,T])$ solution with a sequence of smooth functions in the $C^\theta(\T^3\times [0,T])$ topology, the first natural restriction is to take the complete metric space in which to apply the Baire argument to be a closed subset of $c^{\theta}(\T^3\times[0,T])$. Once one can guarantee the fact that the mollifications of $u_0$ are close in the right topology to $u_0$, the next step is to use the convex integration scheme on a close enough space-time mollification of $u_0$, let us call it $u_\delta$, $\delta > 0$ being the parameter of mollification. Let us moreover denote with $R_\delta$ the Reynold stress tensor of $u_\delta$, i.e.
\[
R_\delta = u_\delta\otimes u_\delta - (u_0\otimes u_0)_\delta.
\]
In order to apply the scheme, one needs to guarantee step $0$ of the inductive estimates, i.e. \eqref{e:R_q_inductive_est},\eqref{e:v_q_inductive_est}, \eqref{e:v_q_0}, \eqref{e:energy_inductive_assumption}. We will now show that, by choosing any $\theta < \beta$ in order to have the $C^\theta(\T^3\times [0,T])$ closeness of the resulting solution to $u_\delta$ (and therefore to $u_0$), \eqref{e:R_q_inductive_est} and \eqref{e:v_q_inductive_est} become impossible to guarantee using the estimates of Proposition \ref{p:moll}. Through these estimates, one wishes to find $\delta > 0$ and $\alpha >0$ for which
\[
\|\mathring R_\delta\|_{0}\lesssim \delta^{2\theta} \le  \delta_{1}\lambda_0^{-3\alpha} \text{ and } \|u_\delta\|_1\lesssim \delta^{\theta - 1} \le M \delta_0^{\sfrac12}\lambda_0.
\]
These relations are anyway incompatible for any $\delta, \alpha > 0$ if 
\begin{equation}\label{rel}
\delta_q = \lambda_q^{-2\beta} = a^{-2\beta b^q}
\end{equation}
for $a, b > 1$. To see this, notice that a solution $\delta$ would need to satisfy also 
\begin{equation}\label{first}
\delta^{2\theta} \lesssim \delta_1 = \lambda_1^{-2\beta}
\end{equation}
Moreover, the estimate on the $C^1$ norm can be rewritten as 
\begin{equation}\label{second}
\delta_0^{-\frac{1}{2(1-\theta)}}\lambda_0^{-\frac{1}{1-\theta}} \lesssim \delta.
\end{equation}
Combining  \eqref{rel}, \eqref{first} and \eqref{second}, one obtains
\[
a^{-\frac{1-\beta}{1-\theta}} \lesssim a^{-b\frac{\beta}{\theta}},
\]
hence that the function $a \mapsto a^{b\frac{\beta}{\theta} -\frac{1-\beta}{1-\theta}}$ is bounded. Since for every $b >1$, one has $b\frac{\beta}{\theta} -\frac{1-\beta}{1-\theta} > 0$ because of the inequality $\theta < \beta$, we find that $a$ can not be taken freely in an open unbounded interval $(a_0,+\infty)$, hence Proposition \ref{p:main} can not possibly be true in this setting. Nonetheless, as it is clearly stated in \cite{BDSV}, we could have found many $C^\beta(\T^3\times [0,T])$ solutions of \eqref{E} $C^\beta(\T^3\times [0,T])$ close to $u_\delta$, for $\beta < \theta$. This is obviously not sufficient for Theorem \ref{t_baire}. This feature of the "$\theta-\beta$ gap" was noticed also in the work \cite{Is17}, to which we refer the reader for interesting discussions. On the other hand, if the starting point $u_0$ can be approximated in the $C^\theta(\T^3\times [0,T])$ topology by more regular solutions, for instance in $C^{\theta'}(\T^3\times [0,T])$, $\theta < \theta'$, then by the previous discussion it becomes clear that we can now start the scheme from these more regular points obtaining the desired estimates in $C^\theta (\T^3\times [0,T])$. This is exactly the reason for introducing the space $X_\theta$.
\\
\\
We conclude this discussion by noting that, even though it could not contain all the $C^\theta(\T^3\times [0,T])$ solutions of \eqref{E}, $X_\theta$ contains many elements. Indeed, by \cite{BDSV}, for every smooth and positive energy profile $e$ and for every $\theta <\theta' <1/3$, we find a weak solution $u \in C^{\theta'}(\T^3\times [0,T])$ of \eqref{E} with $e = e_u$. Since $\theta' > \theta$, $u \in X_\theta$.

\begin{appendix}
\section{Time estimates of Euler equations}
Using the same technique introduced in \cite{CD} to prove the time regularity for H\"older solutions of Euler, we prove the following 
\begin{prop}\label{stimaintempo}
Let $u,v:\T^3\times [0,T]\rightarrow \R^3$ be two weak solutions of \eqref{E} such that $u,v\in C^0(([0,T];C^\theta(\T^3))$ for some $\theta\in (0,1)$. Then there exists a constant $C>0$, depending only on $\theta$, $\|u\|_\theta $ and $\|v\|_\theta$, such that 
$$
\|u-v\|_{C^\theta_{x,t}}\leq C\|u-v\|_\theta.
$$
\end{prop}
\begin{proof}
We define $w:=u-v$. We start by noticing that the  H\"older norm, in the space-time variables, decouples as follows
$$
\frac{|w(x,s)-w(y,t)|}{|(x,s)-(y,t)|^\theta}\leq \frac{|w(x,s)-w(y,s)|}{|x-y|^\theta}+\frac{|w(y,s)-w(y,t)|}{|t-s|^\theta}\leq \|w\|_\theta+\frac{|w(y,s)-w(y,t)|}{|t-s|^\theta}.
$$
Thus it is enough to show that there exists a constant $C>0$, independent of $y,t,s$, such that 
\begin{equation}\label{normasolotempo}
\frac{|w(y,s)-w(y,t)|}{|t-s|^\theta}\leq C\|w\|_\theta.
\end{equation}
If $p$ and $q$ are the corresponding pressures associated to $u$ and $v$ respectively, one has that $w$ solves
\begin{equation}\label{eulero_diff}
\partial_t w + \diver( w\otimes u + v\otimes w)+\nabla(p-q)=0.
\end{equation}
By taking the divergence of \eqref{eulero_diff}, we get
$$
-\Delta(p-q)=\diver \diver (w\otimes u+v\otimes w),
$$
from which, by Schauder estimates, we get 
\begin{equation}\label{stima_pressioni}
\|p-q\|_\theta\leq  \|w\|_\theta \left(\|u\|_\theta+\|v\|_\theta \right)\leq  C \|w\|_\theta.
\end{equation}
Let now $w_\delta=w*\varphi_\delta$ the space mollification of $w$, for some $\delta>0$ that will be fixed at the end of the proof. Since $w\in C^0([0,T];C^\theta(\T^3))$ we have
$$
|w(y,t)-w_\delta(y,t)|\leq C \|w\|_\theta \delta^\theta \qquad \forall t\in [0,T],
$$
from which, by adding and subtracting $w_\delta(y,s)$ and $w_\delta(y,t)$, we can estimate
\begin{equation}\label{holderspazio}
|w(y,s)-w(y,t)|\leq  C \|w\|_\theta \delta^\theta+|w_\delta(y,s)-w_\delta(y,t)|.
\end{equation}
Moreover, since $w$ solves \eqref{eulero_diff}, we get
\begin{equation}\label{w_delta_lip}
|w_\delta(y,s)-w_\delta(y,t)|\leq |t-s| \|\partial_t w_\delta\|_{C^0_{x,t}}\leq  |t-s|\big( \|(w\otimes u+v\otimes w)_\delta\|_1+ \|(p-q)_\delta\|_1\big).
\end{equation}
By estimate \eqref{stima_pressioni} and \eqref{mollest3}, we have
$$
\|(p-q)_\delta\|_1\leq C\|w\|_\theta \delta^{\theta-1},\quad  \forall \delta > 0,
$$
and also
$$
\| (w\otimes u+v\otimes w)_\delta  \|_1\leq C\delta^{\theta-1} \| w\otimes u+v\otimes w\|_\theta\leq C\|w\|_\theta \delta^{\theta-1},\quad  \forall \delta > 0.
$$
Thus, by plugging these two last inequalities in \eqref{w_delta_lip}, we get
$$
|w_\delta(y,s)-w_\delta(y,t)|\leq C |t-s|\delta^{\theta-1}\|w\|_\theta ,\; \forall \delta > 0,
$$
from which, by \eqref{holderspazio}, we conclude 
$$
|w(y,s)-w(y,t)|\leq C (\delta^\theta+|t-s|\delta^{\theta-1})\|w\|_\theta,\quad  \forall \delta > 0.
$$
By choosing $\delta=|t-s|$ we finally achieve \eqref{normasolotempo}, and this concludes the proof.
\end{proof}
\end{appendix}

\end{document}